\documentclass[12pt]{amsart}

\usepackage{amsmath}
\usepackage[curve]{xypic}
\usepackage{caption}
\usepackage{xspace}
\usepackage{cite}
\usepackage{enumitem}
\usepackage{color}
\usepackage{url}

\makeatletter 
\def\@cite#1#2{{\m@th\upshape\bfseries%
[{#1\if@tempswa{\m@th\upshape\mdseries, #2}\fi}]}}
\makeatother 

\theoremstyle{plain}
\newtheorem{thm}{Theorem}[section]
\newtheorem{cor}[thm]{Corollary}
\newtheorem{prop}[thm]{Proposition}
\newtheorem{lem}[thm]{Lemma}
\theoremstyle{definition}

\newtheorem{eg}[thm]{Example}

\newtheorem{conj}[thm]{Conjecture}
\theoremstyle{remark}
\newtheorem{rem}[thm]{Remark}

\numberwithin{equation}{subsection}
\renewcommand{\bold}[1]{\medskip \noindent {\bf #1 }\nopagebreak}

\newcommand{\nc}{\newcommand}
\newcommand{\rnc}{\renewcommand}
\newcommand{\oQ}{\overline{\mathbb{Q}}_\bR}
\newcommand{\bk}{{\mathbf{k}}}
\newcommand{\bg}{{\mathbf{g}}}

\nc\bA{\mathbb{A}}
\nc\bB{\mathbb{B}}
\nc\bC{\mathbb{C}}
\nc\bD{\mathbb{D}}
\nc\bE{\mathbb{E}}
\nc\bF{\mathbb{F}}
\nc\bG{\mathbb{G}}
\nc\bH{\mathbb{H}}
\nc\bI{\mathbb{I}}
\nc{\bJ}{\mathbb{J}} 
\nc\bK{\mathbb{K}}
\nc\bL{\mathbb{L}}
\nc\bM{\mathbb{M}}
\nc\bN{\mathbb{N}}
\nc\bO{\mathbb{O}}
\nc\bP{\mathbb{P}}
\nc\bQ{\mathbb{Q}}
\nc\bR{\mathbb{R}}
\nc\bS{\mathbb{S}}
\nc\bT{\mathbb{T}}
\nc\bU{\mathbb{U}}
\nc\bV{\mathbb{V}}
\nc\bW{\mathbb{W}}
\nc\bY{\mathbb{Y}}
\nc\bX{\mathbb{X}}
\nc\bZ{\mathbb{Z}}
\nc\cA{\mathcal{A}}
\nc\cB{\mathcal{B}}
\nc\cC{\mathcal{C}}
\rnc\cD{\mathcal{D}}
\nc\cE{\mathcal{E}}
\nc\cF{\mathcal{F}}
\nc\cG{\mathcal{G}}
\rnc\cH{\mathcal{H}}
\nc\cI{\mathcal{I}}
\nc{\cJ}{\mathcal{J}} 
\nc\cK{\mathcal{K}}
\rnc\cL{\mathcal{L}}
\nc\cM{\mathcal{M}}
\nc\cN{\mathcal{N}}
\nc\cO{\mathcal{O}}
\nc\cP{\mathcal{P}}
\nc\cQ{\mathcal{Q}}
\rnc\cR{\mathcal{R}}
\nc\cS{\mathcal{S}}
\nc\cT{\mathcal{T}}
\nc\cU{\mathcal{U}}
\nc\cV{\mathcal{V}}
\nc\cW{\mathcal{W}}
\nc\cY{\mathcal{Y}}
\nc\cX{\mathcal{X}}
\nc\cZ{\mathcal{Z}}

\nc{\dmo}{\DeclareMathOperator}
\rnc{\Re}{\operatorname{Re}}
\rnc{\Im}{\operatorname{Im}}
\dmo{\rank}{rank}
\dmo{\End}{End}
\dmo{\Jac}{Jac}
\dmo{\Id}{Id}
\dmo{\lcm}{lcm}

\nc{\Tm}{Teichm\"uller\xspace}

\begin{document}

\title[Field of definition]{The field of definition of affine invariant submanifolds of the moduli space of abelian differentials}
%
\author[A.Wright]{Alex~Wright}
\address{Math\ Department\\University of Chicago\\
5734 South University Avenue\\
Chicago, IL 60637}
\email{alexmwright@gmail.com}
%

\begin{abstract}
The field of definition of an affine invariant submanifold $\cM$ is the smallest subfield of $\bR$ such that $\cM$ can be defined in local period coordinates by linear equations with coefficients in this field. We show that the field of definition is equal to the intersection of the holonomy fields of translation surfaces in $\cM$, and is a real number field of degree at most the genus. 

We show that the projection of the tangent bundle of $\cM$ to absolute cohomology $H^1$ is simple, and give a direct sum decomposition of $H^1$ analogous to that given by M\"oller in the case of Teichm\"uller curves. 

Applications include explicit full measure sets of translation surfaces whose orbit closures are as large as possible, and evidence for finiteness of algebraically primitive Teichm\"uller curves. 

The proofs use recent results of Avila, Eskin, Mirzakhani, Mohammadi, and M\"oller.
\end{abstract}

\maketitle
\thispagestyle{empty}



\section{Introduction}\label{S:intro}
\subsection{Background.}

During the past three decades, it has been discovered that many properties of a translation surface are determined to a  surprising extent by its $SL(2,\bR)$--orbit closure. The orbit closure is relevant for: 
\begin{itemize}
\item The dynamics of the straight line flow, for example the Veech dichotomy and deviations of ergodic averages \cite{V,F, EKZbig};
\item Counting problems (the Siegel-Veech formula) \cite{V4, EMa};
\item Flat geometry, for example which translation surfaces can be written as a convex polygon with edge identifications \cite{V5}.
\end{itemize}

Recently Eskin-Mirzakhani-Mohammadi have announced a proof that the $SL(2,\bR)$--orbit closure of every unit area translation surface is the set of unit area translation surfaces in some \emph{affine invariant submanifold} \cite{EM, EMM} (definitions are recalled in Section 2). However, the work of Eskin-Mirzakhani-Mohammadi does not give the orbit closure of any particular translation surface.  Prior to this work, the orbit closure was known for only a measure zero subset of translation surfaces of genus greater than two.  

\subsection{Statement of results.} 

The \emph{field of definition} $\bk(\cM)$ of an affine invariant submanifold $\cM$ is the smallest subfield of $\bR$ such that $\cM$ can be defined in local period coordinates by linear equations with coefficients in this field.

\begin{thm}\label{T:main}
The field of definition $\bk(\cM)$ of an affine invariant submanifold $\cM$ is a real number field of degree at most the genus. It is equal to the intersection of the holonomy fields of all translations surfaces in $\cM$. 
\end{thm}  

The second statement gives that $\bk(\cM)$ can be explicitly calculated from the absolute periods of translation surfaces in $\cM$. 

\begin{eg}
If $\cM$ contains a single square-tiled surface, or even just a single translation surface whose absolute periods are in $\bQ[i]$ and whose relative periods are arbitrary, then $\bk(\cM)=\bQ$. 
\end{eg}

\bold{Generic translation surfaces.} A unit area translation surface is said to be \emph{$\cM$--generic} if its orbit closure is equal to $\cM_1$, the set of unit area translation surfaces within $\cM$. Work of Masur and Veech \cite{Ma2, V2} gives the ergodicity of the $SL(2,\bR)$--action on $\cM_1$, and ergodicity guarantees that almost every point (with respect to a natural invariant smooth measure) in $\cM_1$ is $\cM$--generic. (Masur and Veech worked in a less general setting, but their proofs apply equally well to affine invariant submanifolds.) However, prior to this work not many examples of \emph{explicit} generic translation surfaces were known, see below for a summary. 

In Section \ref{S:gen}, we define what it means for a translation surface to have \emph{$\cM$--typical periods}. Roughly, $(X,\omega)$ has $\cM$--typical periods if its periods do not satisfy any linear equation which might, according to Theorem \ref{T:main}, define an affine invariant submanifold properly contained in $\cM$. Having $\cM$--typical periods is an explicit, field theoretic condition. 

\begin{cor}\label{C:generic}
Let $\cM$ be any affine invariant submanifold. Let $\cG$ be the set of translation surfaces $(X,\omega)\in\cM_1$ with $\cM$--typical periods. Then
\begin{enumerate}
\item $\cG$ has full measure in $\cM_1$, and 
\item every translation surface in $\cG$ is $\cM$--generic. 
\end{enumerate}
\end{cor}

Let $\overline{\bQ}$ denote the algebraic closure of $\bQ$, and set $\oQ= \overline{\bQ}\cap\bR$. The set $\cG$ includes in particular the translation surfaces in $\cM_1$ whose period coordinates span a $\oQ$--vector space of dimension $\dim_\bC \cM$. So in particular, Corollary \ref{C:generic} implies that translation surfaces whose period coordinates are ``sufficiently transcendental" have full orbit closure. 

\begin{eg}\label{eg1}
Consider a translation surface $(X,\omega)$ in a stratum $\cH$, and assume that the real parts of the period coordinates of $(X,\omega)$ are contained in $\bQ[\pi, \pi^2, \pi^3, \ldots]$ and are linearly independent over $\bQ$. Then the period coordinates of $(X,\omega)$ are automatically linearly independent over $\oQ$, and so $(X,\omega)$ has $\cH$--typical periods. Corollary \ref{C:generic} gives that such $(X,\omega)$ are $\cH$--generic. 
\end{eg}

The set $\cG$ also contains a dense set of translation surfaces whose periods lie in a number field.

\bold{Global structure.} Our results on field of definition in fact follow from considerations about the global structure of affine invariant submanifolds which are of independent interest.

 Let $H^1$ denote the flat bundle over $\cM$ whose fiber over $(X, \omega)\in \cM$ is $H^1(X,\bC)$, and let $H^1_{rel}$ denote the flat bundle whose fiber over $(X,\omega)$ is $H^1(X,\Sigma, \bC)$, where $\Sigma$ is the set of singularities of $(X,\omega)$. Let $p:H^1_{rel}\to H^1$ denote the natural projection from relative to absolute cohomology. Note that $T(\cM)$ is a flat subbundle of $H^1_{rel}$.

The field of definition of a flat subbundle $E\subset H^1$ is the smallest subfield of $\bR$ so that locally the linear subspace $E$ of $H^1(X,\bC)$ can be defined by linear equations (with respect to an integer basis of $H_1(X, \bZ)$) with coefficients in this field.  The trace field of a flat bundle over $\cM$ is defined as the field generated by traces of the corresponding representation of $\pi_1(\cM)$.

\begin{thm}\label{T:galois}
Let $\cM$ be an affine invariant submanifold. The field of definition of $p(T(\cM))$ and trace field of $p(T(\cM))$ are both equal to $\bk(\cM)$. 

Set $\bV_{\Id}=p(\cT(\cM))$. There is a semisimple flat bundle $\bW$, and for each field embedding $\rho:\bk(\cM)\to\bC$ there is a flat simple bundle $\bV_\rho$ which is Galois conjugate to $\bV_{\Id}$, so that
\[H^1 = \left(\bigoplus_\rho \bV_\rho\right) \oplus \bW.\]

The bundle $\bW$ does not contain any subbundles isomorphic to any $\bV_\rho$. Both $\bW$ and $\oplus \bV_\rho$ are defined over $\bQ$. 

In particular, 
\[\dim_\bC p(T(\cM)) \cdot \deg_\bQ \bk(\cM) \leq 2g.\]
\end{thm}

The direct sum decomposition of $H^1$ in Theorem \ref{T:galois} was previously established in the case of Teichm\"uller curves by Martin M\"oller \cite{M}, and is one of the main tools used in the study of closed $SL(2,\bR)$--orbits. In the Teichm\"uller curve case, M\"oller showed that the splitting of $H^1$ is compatible with the Hodge decomposition; this is conjectured  for general affine invariant submanifolds, but this does not follow from our work.\footnote{Added in proof: This has been established by Simion Filip \cite{Fi1, Fi2}, who has also shown that the field of definition is totally real and that affine invariant submanifolds are  varieties. } When the splitting of $H^1$ is non-trivial ($H^1\neq p(T(\cM))$) and compatible with the Hodge decomposition, then $\cM$ parameterizes translation surfaces whose Jacobians admit non-trivial endomorphisms \cite[Lemma 4.2]{M6}.

There may be a even more direct connection between the field of definition and the global structure of affine invariant submanifolds. 

\begin{conj}[Mirzakhani]
If an affine invariant submanifold $\cM$ is defined over $\bQ$, and $\cM$ is not a connected component of a stratum, then every translation surface in $\cM$ covers a quadratic differential (half-translation surface) of smaller genus. 
\end{conj} 

\begin{conj}[Mirzakhani]
If an affine invariant submanifold $\cM$ is not defined over $\bQ$, then $p(T(\cM))$ has dimension $2$.
\end{conj}

In the case that $p(T(\cM))$ is 2 dimensional, $\cM$ should parameterize eigenforms for real multiplication. (This real multiplication may only be present on a factor of the Jacobian up to isogeny, instead of the entire Jacobian. This is analogous to the case of Teichm\"uller curves treated in \cite{M}.)

%

\bold{Evidence for finiteness of algebraically primitive Teichm\"uller curves.} Algebraically primitive Teichm\"uller curves correspond to two complex dimensional affine invariant submanifolds where the trace field of $p(T(\cM))$ has degree equal to the genus.

The equidistribution results of Eskin-Mirzakhani-Mohammadi \cite{EMM} and Theorem \ref{T:galois} allow us to show

\begin{thm}\label{T:primequi}
Let $\cH$ be a connected component of the minimal stratum in prime genus. If $\cH$ contains infinitely many algebraically primitive Teichm\"uller curves, then some subsequence equidistributes towards $\cH$.
\end{thm}

We hope that Theorem \ref{T:primequi} will eventually lead to a proof that there are only finitely many algebraically primitive Teichm\"uller curves in the minimal stratum in prime genus greater than two.\footnote{Added in proof: This hope has been realized in joint work with Matheus \cite{MW}. Bainbridge and M\"oller have informed the  author that together with Habegger they have recently established new finiteness results that are complementary to those in \cite{MW}, using  different methods.} 


\bold{Countability of affine invariant submanifolds.} Theorem \ref{T:main} also provides an alternate proof of a step in Eskin-Mirzakhani-Mohammadi's theorem on $SL(2,\bR)$--orbit closures.

\begin{cor}[\cite{EMM}, Prop. 2.16]
There are only countably many affine invariant submanifolds in the moduli of translation surfaces. 
\end{cor}

\begin{proof}
There are only countably many systems of real linear equations all of whose coefficients lie in a number field. 
\end{proof}

\subsection{Previous results}

Here we list the previous results which form the context for our applications to problems predating \cite{EM, EMM}. We do not use any of these results, and so omit some definitions which the reader can find in the references. For an introduction to the field of translation surfaces see the surveys \cite{MT, Z}. 

\bold{Generic translation surfaces.} In the case that $\cM$ is a connected component of a stratum, we simply call an $\cM$-generic translation surface \emph{generic}. 

McMullen has classified orbit closures in genus 2 (strata of abelian differentials).

\begin{thm}[\cite{Mc5}, Thm. 1.2]\label{T:g2}
The $SL(2,\bR)$--orbit closure of any unit area translation surface in $\cH(2)$ or $\cH(1,1)$ is either 
\begin{enumerate}
\item equal to a closed orbit, or 
\item equal to a locus of $(X,\omega)$ of unit area, where $\Jac(X)$ admits real multiplication with $\omega$ as an eigenform, or 
\item equal to the whole stratum. 
\end{enumerate}
\end{thm}

All of the possible orbit closures (affine invariant manifolds) listed had been previously studied by both McMullen and Calta \cite{Mc, Ca}. Calta describes these orbit closures in terms of the $J$--invariant, and gives explicit linear equations defining the affine invariant manifolds \cite{Ca}.

It follows directly from this theorem of McMullen that in particular any translation surface in $\cH(2)$ or $\cH(1,1)$ whose absolute periods do not satisfy a linear relation with coefficients in $\bQ[\sqrt{d}]$ for some $d\geq 1$ is generic. McMullen's result shows that Corollary \ref{C:generic} is not sharp even in genus 2.  

Let $\cL$ be the locus of hyperelliptic translation surfaces in $\cH(2,2)$, where both of the singularities are fixed by the hyperelliptic involution. Examples of generic translation surfaces have been constructed in  $\cL$ by Hubert-Lanneau-M\"oller \cite{HLM-AY, HLM-Q, HLM-S} and also in $\cH^{hyp}(4)$ by Nguyen \cite{N}. These results in $\cL$ and $\cH ^{hyp}(4)$ are complementary to ours; they do not provide a full measure set, but do provide many especially interesting and important examples not covered by our results.

\begin{thm}[\cite{HLM-S}, Thm 0.2]
Suppose $(X,\omega)\in\cL$ is obtained by the Thurston-Veech construction, has cubic trace field, and has a completely periodic direction that is not parabolic. Then $(X,\omega)$ is $\cL$--generic. 
\end{thm}

In particular, there are $\cL$--generic translation surfaces whose periods lie in a cubic field (that is, there is a real cubic field $\bk$ so that the period coordinates lie in $\bk[i]$). In private communication Erwan Lanneau has indicated to the author that some of the results in \cite{HLM-S} can be extended to certain strata in higher genus.

\begin{thm}[\cite{N}, Thms 1.1, 1.2, Cor. 1.3]
Every translation surface in $\cH^{hyp}(4)$ admits a specific decomposition into parallelograms and cylinders. When two edges are assumed to be parallel, then for an explicit generic subset of the remaining parameters, the translation surface is generic. In particular, in $\cH^{hyp}(4)$ there are generic surfaces arising from the Thurston-Veech construction, and generic surfaces with period coordinates in a quadratic field.

Any surface in $H^{hyp}(4)$ with a completely periodic direction consisting of three cylinders whose moduli are independent over $\bQ$ is generic. 
\end{thm}

These results of McMullen, Hubert-Lanneau-M\"oller and Nguyen all rely on explicit decompositions of the translation surface into simple pieces such as tori. Sufficiently simple decompositions are not available for the generic translation surface in high genus. 

In the case that $\cM$ is two complex dimensional (i.e., corresponds to a closed $SL(2,\bR)$--orbit), then every translation surface in $\cM$ is $\cM$--generic. See \cite{W2} for a list of known closed $SL(2,\bR)$--orbits; additional examples arise from covering constructions. 


\bold{Algebraically primitive Teichm\"uller curves.} There are infinitely many algebraically primitive Teichm\"uller curves in $\cH(2)$, which were constructed by McMullen and Calta \cite{Ca, Mc}. McMullen showed that there is only one algebraically primitive Teichm\"uller curve in $\cH(1,1)$ \cite{Mc4}.

Finiteness of algebraically primitive Teichm\"uller curves is known in $\cH^{hyp}(g-1,g-1)$ by work of M\"oller \cite{M3}, and in $\cH(3,1)$ by work of Bainbridge-M\"oller \cite{BaM}. 

\subsection{Tools and motivation.} That affine invariant submanifolds are defined over number fields is easier than the other statements of Theorem \ref{T:main}. This can be proved using only a Closing Lemma for Teichm\"uller geodesic flow. We sketch this approach briefly in Section \ref{S:closed}. The result also follows, together with the bound on the degree of the field of definition, from the inequality in Theorem \ref{T:galois}. 

Both Hamenst\"adt \cite{Ha} and Eskin-Mirzakhani-Rafi \cite{EMR} have proven Closing Lemmas for Teichm\"uller geodesic flow. The version we require is extremely close to that given in \cite{EMR}. The Closing Lemma is also used in the proof of Theorem \ref{T:galois}. 

Theorem \ref{T:galois} also uses a result of Avila-Eskin-M\"oller.

\begin{thm}[\cite{AEM}, Theorem 1.5]\label{T:AEM}
The bundle $H^1$ over $\cM$ is semisimple. That is, every flat subbundle has a flat complement.
\end{thm}

In fact, Avila-Eskin-M\"oller prove this for the bundle $H^1_\bR$ whose fiber over $(X,\omega)$ is $H^1(X,\bR)$. The result follows for $H^1=H^1_\bR \otimes_\bR \bC$ from general principles, see Section \ref{S:simple}.

That affine invariant submanifolds are defined over a number field is motivated by the expectation that they have a fairly rigid algebro-geometric structure. The use of the Closing Lemma in the proof is motivated by the fact that, up to scaling, the real and imaginary parts of a translation surface on a closed orbit for the Teichm\"uller geodesic flow have period coordinates in a number field.   

That the field of definition of $p(T(\cM))$ is equal to that of $T(\cM)$ is motivated by the fact that for any $(X,\omega)$ having a hyperbolic affine diffeomorphism, and in particular any $(X,\omega)$ lying on a closed $SL(2,\bR)$--orbit, the absolute and relative periods span the same $\bQ$--vector subspace of $\bC$  \cite{KS, Mc6}. 

\bold{Organization.} Section \ref{S:affine} contains definitions, and Section \ref{S:gen} proves Corollary \ref{C:generic} using Theorem \ref{T:main}. Section \ref{S:closed} discusses the Closing Lemma, which is used to prove the ``Simplicity Theorem" in Section \ref{S:simple}. The Simplicity Theorem includes the statement that $p(T(\cM))$ is simple from Theorem \ref{T:galois}.  The remainder of Theorem \ref{T:galois} is established in Section \ref{S:mon}, using the Simplicity Theorem and the Closing Lemma. Theorem \ref{T:main} is proven in Section \ref{S:hol}, using Theorem \ref{T:galois}. Theorem \ref{T:primequi} is covered in the final section. 

\bold{Acknowledgements.} Special thanks go first of all to the author's thesis advisor Alex Eskin, for inspiring this work and generously contributing ideas at several stages. Special thanks also go to Maryam Mirzakhani for sharing her conjectures and allowing us to reproduce them here, and to Jon Chaika for helpful conversations during the early stages of this work.

The author is grateful to Jayadev Athreya, David Aulicino, Matt Bainbridge, Simion Filip, Ursula Hamenst\"adt, Erwan Lanneau, Howard Masur, Carlos Matheus, Martin M\"oller, Ronen Mukamel, Duc-Manh Nguyen, Kasra Rafi, John Smillie, Anton Zorich and many other members of the translation surfaces community for helpful and enjoyable conversations, and in some cases also for feedback on earlier drafts. The author is grateful to the referee for helpful comments.

Part of this research was conducted while the author was partially supported by the National Science and Engineering Research Council of Canada.


\section{Period coordinates and field of definition}\label{S:affine}

Suppose $g\geq 1$ and let $\alpha$ be a partition of $2g-2$. The stratum $\cH(\alpha)$ is defined to be the set of $(X,\omega)$ where $X$ is a genus $g$ closed Riemann surface, and $\omega$ is a holomorphic 1-form on $X$ whose zeroes have multiplicities given by $\alpha$. In fact, we immediately replace $\cH(\alpha)$ by a finite cover $\cH$ which is a manifold instead of an orbifold.

Given a translation surface $(X,\omega)$, let $\Sigma\subset X$ denote the set of zeros of $\omega$. Pick any basis $\{\xi_1, \ldots, \xi_n\}$ for the relative homology group $H_1(X,\Sigma; \bZ)$. The map $\Phi:\cH\to \bC^n$ defined by 
\[\Phi(X,\omega)=\left( \int_{\xi_1} \omega, \ldots, \int_{\xi_n}\omega \right)\]
defines local \emph{period coordinates} on a neighborhood $(X,\omega)\in \cH$. We may also refer to the \emph{absolute period coordinates} of a translation surface: these are the integrals of $\omega$ over a basis of absolute homology $H_1(X,\bZ)$. 

Period coordinates provide $\cH$ with a system of coordinate charts with values in $\bC^n$ and transition maps in $SL(n,\bZ)$. An \emph{affine invariant submanifold} of $\cH$ is  an immersed manifold $\cM\hookrightarrow \cH$ such that each point of $\cM$ has a neighborhood whose image is locally defined by real linear equations in  period coordinates. The possibility that $\cM$ might be immersed instead of embedded will not cause us any problems. (There are two ways to see that immersions do not cause problems for our arguments. First, one can simply phrase the arguments in $\cM$ instead of in its image in the stratum. This involves only notational changes. Second, \cite{EMM} gives that $\cM$ is embedded away from a closed locus of measure zero, and this locus can be avoided in our arguments.) We will typically treat $\cM$ as embedded for notational simplicity. All linear equations in this paper are assumed to be homogeneous, i.e. have constant term 0.

Given any subspace $V\subset \bC^n$, one can define its \emph{field of definition} as the unique subfield $\bk\subset \bC$ so that $V$ can be defined by linear equations with coefficients in $\bk$ but $V$ cannot be defined by linear equations with coefficients in a proper subfield of $\bk$. We require that the variables defining these linear equations are the coordinates of $\bC^n$.

\begin{lem}
The field of definition of any subspace $V\subset \bC^n$ is well-defined, and furthermore it is the smallest subfield of $\bC$ such that $V$ is spanned by vectors with coordinates in this subfield. 
\end{lem}

\begin{proof}
Left to the reader. In fact the field of definition is the field generated by the coefficients in the row reduced echelon form of any system of linear equations defining $V$.
\end{proof}

Given an affine invariant submanifold $\cM$, we define its \emph{field of definition} $\bk(\cM)$ to be the field of definition of the linear subspace of $\bC^n$ that defines $\cM$ in period coordinates. Since the transition maps are in $SL(n,\bZ)$, this does not depend on which coordinate chart is used. 

Given a translation surface $(X,\omega)$, its \emph{holonomy} is the subset $\Lambda\subset \bC$ of the holonomies of all saddle connections. (A saddle connection is a straight line between two singularities on the translation surface, and the holonomy of a saddle connection is the integral of $\omega$ over the relative homology class of this line.) Similarly the \emph{absolute holonomy} is the subset $\Lambda_{abs}\subset \bC$ of the holonomies of all saddle connections representing absolute homology classes.

Given $e_1, e_2 \in \Lambda_{abs}$ linearly independent over $\bR$, the \emph{holonomy field} of $(X,\omega)$ is the smallest subfield $\bk\subset \bR$ so that $\Lambda_{abs}\subset \bk e_1 \oplus \bk e_2$. This does not depend on $e_1, e_2$. The definition of holonomy field is due to Kenyon-Smillie \cite{KS}. 

Note in particular, that for some $A\in GL(2,\bR)$, we have $Ae_1=1$ and $Ae_2=i$. Hence $A\Lambda_{abs} \subset \bk[i]$, and in particular all of the absolute period coordinates of $A(X,\omega)$ lie in $\bk[i]$. 


\section{Generic translation surfaces}\label{S:gen}

In this section we prove Corollary \ref{C:generic} assuming Theorem \ref{T:main}. First we must give the definitions. 

We say that a translation surface $(X,\omega)\in \cM_1$ has \emph{$\cM$--special periods} if there is some subfield $\bk$ of the holonomy field of $(X,\omega)$, such that 
\begin{enumerate}
\item $\bk$ has degree at most the genus $g$ of $(X,\omega)$, and 
\item $\bk(\cM)\subset \bk$, and
\item there is some linear equation on local period coordinates with coefficients in $\bk$ which does not hold identically on $\cM$, but nonetheless holds at $(X,\omega)$. 
\end{enumerate}

A translation surface is said to have \emph{$\cM$--typical periods} if it does not have $\cM$--special periods.

\begin{proof}[\textbf{\emph{Proof of Corollary \ref{C:generic}}}]
In local period coordinates, $\cM$ is some $\dim_\bC \cM$--dimensional subspace $V$ of $\bC^n$. The unit area translation surfaces in $V$ form a hypersurface, and the natural invariant volume on $\cM_1$ is up to scaling the disintegration of Lebesgue measure on $V$ \cite{EM}. 

There are only countably many linear equations on $V$ which are nonzero on $V$ and which have coefficients in a number field of degree at most $g$. Their union is measure zero, and their complement is contained in $\cG$ in local period coordinates. Hence $\cG$ has full measure. 

Consider any translation surface $(X,\omega)$ in $\cG$. Its orbit closure is some affine invariant manifold $\cN\subset \cM$, which is defined in local period coordinates by linear equations with coefficients in $\bk=\bk(\cN)$. By Theorem \ref{T:main}, $\bk(\cN)$ is contained in the holonomy field of $(X,\omega)\in \cN$, and is a field of degree at most $g$. Furthermore, since $\bk(\cM)$ is the intersection of the holonomy fields of surfaces in $\cM$, and $\bk=\bk(\cN)$ is the intersection of the holonomy field of surfaces in $\cN\subset \cM$, we see that $\bk\supset \bk(\cM)$. 

However, the coordinates of $(X,\omega)$ by assumption do not satisfy any linear equations with coefficients in such a field $\bk$ that do not hold identically on $\cM$. Hence $\cN=\cM$.  
\end{proof}


\section{Closed orbits for the Teichm\"uller geodesic flow}\label{S:closed}

The Teichm\"uller geodesic flow is given by 

\[g_t=\left(\begin{array}{cc} e^t& 0 \\0& e^{-t}\end{array}\right)\subset SL(2,\bR).\]

\subsection{Closed orbits are abundant.} The following Closing Lemma is due to Eskin-Mirzakhani-Rafi \cite[Section 9]{EMR} using results of Forni. See also the Closing Lemma of Hamenstadt  \cite[Section 4]{Ha}. 

 \begin{lem}[Closing Lemma] 
Let $\cK$ be  an arbitrary compact subset of an affine invariant submanifold $\cM$. Given any open set $U'\subset \cM$ that intersects $\cK$, there is a smaller open set $U\subset U'$ and a constant $L_0>0$ with the following property. 

Assume that $\bg:[0,L]\to \cM$ is a segment of a $g_t$--orbit (parameterized by $t$) such that the following three conditions hold:
\begin{enumerate}
\item $\bg(0), \bg(L)\in U$, 
\item $L>L_0$, 
\item $|\{t\in [0,L]| \bg(t)\in \cK\}| > L/2$. 
\end{enumerate}
Then there exists a closed $g_t$--orbit that intersects $U'$. 
\end{lem} 

\begin{proof}
The proof follows as in \cite[Section 9]{EMR}. They discuss only strata of quadratic differentials, but the key feature of local product structure (expanding and contracting foliations) is guaranteed by real linearity for any affine invariant submanifold. (The proof proceeds by applying the contraction mapping principal to the stable and unstable foliations. Thus, only the hyperbolicity of the flow is relevant. The flow ``accumulates  hyperbolicity" at a definite rate as it spends time in a fixed compact set, which is why condition 3 is required.)  
\end{proof}

In fact stronger statements are true; see the references. We have stated the Closing Lemma as is because we only need the existence of closed orbits. 

\subsection{The action on (relative) cohomology.}\label{SS:act} If $g_T(X,\omega)=(X,\omega)$, then there is an induced pseudo-Anosov diffeomorphism on $X$. The action $g_T^*$ on the relative and absolute cohomology of $X$ is via this pseudo-Anosov. 

The action of $g_T^*$ on absolute cohomology is related to its action on relative cohomology $H^1(X,\Sigma, \bC)$, where $\Sigma$ is the set of singularities. After replacing the pseudo-Anosov with a power, we may assume it fixes $\Sigma$ pointwise. In this case the pseudo-Anosov acts trivially on $\ker(p)$, and thus the action on $H^1(X,\Sigma, \bC)$ has block triangular form, where the diagonal blocks are the identity (on $\ker(p)$) and the action on absolute cohomology. 

Both $e^{-T}$ and $e^{T}$ are simple eigenvalues for $g_T^*$ \cite{Fr, FLP, Pe}, and the eigenvectors are $\Re(\omega)$ and $\Im(\omega)$. The abelian differential defines both a cohomology class and a relative cohomology class, and this statement is true in both cohomology and relative cohomology. Furthermore, we emphasize that the generalized eigenspaces of $e^{-T}$ and $e^{T}$ are one-dimensional in both cohomology and relative cohomology. 

\begin{rem}
We may now sketch a proof that every affine invariant submanifold $\cM$ is defined over a number field. Details are left to the reader, as this will be reproved in a stronger form later. 

Pick any period coordinate chart for $\cM$. It is not hard using the Closing Lemma and ergodicity to show that closed orbits are dense in $\cM_1$. ($\cK$ can chosen to be any compact set with measure greater than $0.5$, and condition 3 is verified using the Birkhoff Ergodic Theorem by letting the starting point be generic.) If $(X,\omega)$ lies on a closed orbit, then the real and imaginary parts of the period coordinates lie in a number field up to scaling, because they are eigenvectors of simple eigenvalues of an integer matrix. The integer matrix is $g_T^*:H^1(X,\Sigma, \bZ)\to H^1(X,\Sigma, \bZ)$.

The result now follows from the fact that any subspace $V\subset \bC^n$ which is spanned by points with coordinates in a number field is defined over a number field. 
\end{rem}


\section{The Simplicity Theorem}\label{S:simple}

In this section we prove that $p(T(\cM))$ is simple, using Avila-Eskin-M\"oller's Theorem (Theorem \ref{T:AEM}). First we address the difference between the statements of Theorem \ref{T:AEM} and \cite[Theorem 1.5]{AEM}.

\begin{proof}[\emph{\textbf{Proof of Theorem \ref{T:AEM}}}]
The difference between the statements is simply the difference between using real and complex cohomology. 

Flat bundles correspond to representations of the fundamental group of the base.  \cite[Theorem 1.5]{AEM} gives that the representation to real cohomology is the sum of real representations without real invariant subspaces. The complexification of any real representation without invariant subspaces is either simple or the direct sum of two complex conjugate simple representations. 
\end{proof}

In the following statement, $p(T(\cM))$ and $T(\cM)$ are regarded as  flat complex vector bundles over $\cM$.

\begin{thm}[The Simplicity Theorem]\label{T:simple}
The  bundle $p(T(\cM))$ is simple. Furthermore, $T(\cM)$  cannot be expressed as a direct sum of flat complex vector subbundles, and any proper flat subbundle of $T(\cM)$ is contained in $\ker(p)$.
\end{thm}

\begin{proof}
By Theorem \ref{T:AEM}, to show $p(T(\cM))$ is simple it suffices to show that it cannot be written nontrivially as a direct sum $p(T(\cM))=E'\oplus E''$ of two flat bundles. Suppose in order to find a contradiction that such a direct sum decomposition exists. 

Pick a small open set $U'\subset \cM$ such that, for all $v\in U'$, if we set $w=p(v)$, and write $w=w'+w''$ with $w'\in E'$ and $w''\in E''$, then the real and imaginary parts of both $w'$ and $w''$ are nonzero.  (Locally, $\cM$ looks like an open subset of a complex vector space, and the set of $v$ for which the real or imaginary part of $w'$ or $w''$ vanish is a union of four real vector subspaces.) Apply the Closing Lemma to find $v\in U'$ which generates a closed $g_t$--orbit. 

Let $\phi$ be the mapping class of $g_{T}$, where $g_{T}v=v$. Then $\phi$ is a pseudo-Anosov, and 
\[\Re(w)=\Re(g_{T}w)= e^T \phi^* \Re(w).\]
 By assumption both $w'$ and $w''$ are nonzero, and we have 
 \[\Re(w')= e^T \phi^* \Re(w') \quad \quad \text{and}\quad \quad \Re(w'')= e^T \phi^* \Re(w'').\]

However, $e^{T}$ is the unique largest eigenvalue for $(\phi^*)^{-1}$, and it is a simple eigenvalue. This contradicts the fact that we have found two eigenvectors for this eigenvalue. Hence $p(T(\cM))$ is simple. 

The same proof shows that $T(\cM)$ has no direct sum decomposition. Now we wish to show that any proper flat subbundle $B\subset T(\cM)$ is contained in $\ker(p)$. Suppose not. We must have $p(B)=p(T(\cM))$, since $p(B)$ is a nonzero flat subbundle of $p(T(\cM))$, and $p(T(\cM))$ is simple. 

Pass to a finite cover of $\cM'$ for which monodromy does not permute the singularities of flat surfaces. This passage to a finite cover does not effect any of our previous results; already we were working in an unspecified finite cover of a stratum. Lift $B$ to a flat bundle $B'$ over $\cM'$.

Now we have that $\ker(p)$ is a flat subbundle on which monodromy acts trivially. Over any fiber, pick a complement $B''\subset \ker(p)$ of $B'$. Because $\ker(p)$ is a trivial bundle, this complement extends to a flat bundle which is a complement to $B'$. This contradicts that $T(\cM)$ has no direct sum decompositions. 
\end{proof}


\section{Monodromy and Galois action}\label{S:mon}

Part of Theorem \ref{T:galois}, namely that $p(T(\cM))$ is simple, has already been established in the previous section as part of the Simplicity Theorem. In this section we use the Simplicity Theorem to establish the remainder of Theorem \ref{T:galois}.

Fix an affine invariant submanifold $\cM$. The flat bundle $H^1$ over $\cM$ corresponds to a conjugacy class of representations 
\[\pi_1(\cM) \to \End\bC^{2g} \simeq \End H^1(X, \bC)\]
and Theorem \ref{T:AEM} of Avila-Eskin-M\"oller gives that the representation is semisimple.

Furthermore, the representation comes from an integer valued representation  
\[\pi_1(\cM) \to \End\bZ^{2g} \simeq \End H^1(X, \bZ).\] 

Flat subbundles of $H^1$ correspond to subrepresentations, and the Simplicity Theorem gives that the projection of the tangent bundle $p(T(\cM))$ corresponds to an \emph{irreducible} subrepresentation, which we will call $V$. We will also consider the representation $V'$ corresponding to $T(\cM)$. 

\begin{prop}\label{P:galois}
For any affine invariant manifold, let $\bk$ be the trace field of $V$. Then $\bk=\bk(\cM)$. Furthermore, $\bk(\cM)$ is equal to the field of definition of $p(T(\cM))\subset H^1$. 
\end{prop}

The proof has been divided into lemmas, the first of which uses the following basic fact \cite[Cor. 3.8]{Lang}.

\begin{lem}\label{L:bour}(Bourbaki)
Let $G$ be any group (not necessarily finite). Let $V_1$ and $V_2$ be two finite dimensional semisimple representations of $G$ over a field of characteristic zero. Then $V_1$ and $V_2$ are isomorphic as representations if and only if their characters are equal.   
\end{lem}

\begin{lem}
There is a monodromy matrix $A$ of the flat bundle $H^1$ for which the following is true. The matrix $A$ has a simple eigenvalue $\lambda$, and for every field automorphism $\sigma:\bC\to\bC$ which acts as the identity on the trace field of $V$, $\sigma(\lambda)$ is a simple eigenvalue of $A$, and the eigenvector lies in $V$. 

The same statements remain true when $H^1$ is replaced by $H^1_{rel}$, and the matrix $A$ is replaced by the corresponding monodromy matrix $A'$ for  $H^1_{rel}$: all the $\sigma(\lambda)$ are simple eigenvalues of $A'$, and the eigenvectors lie in $V'$. 
\end{lem}

We emphasize that by simple eigenvalue, we mean that the generalized eigenspace is one-dimensional. 

\begin{proof}
Suppose $(X,\omega)\in \cM$ generates a closed $g_t$--orbit, so $g_T(X,\omega)=(X,\omega)$, and let $A$ be the action of $g_T$ on cohomology. In other words, $A$ is a monodromy matrix for the bundle $H^1$. Let $\lambda$ be the eigenvalue of maximal modulus of $A$. We have already observed that $\lambda$ is unique and simple (Section \ref{SS:act}). Furthermore we have observed that the eigenvector of $\lambda$ lies in $V$. 

By Lemma \ref{L:bour} every Galois automorphism that acts as the identity on the trace field must send the irreducible representation $V$ to a conjugate (isomorphic) irreducible representation; that is, these Galois automorphisms send the representation $V$ to itself.

Hence, for any $\sigma$ as in the Lemma, it follows that $\sigma(A|_{V})$ is similar to the matrix $A|_{V}$ (this is the matrix $A$ restricted to the subspace $V$), and consequently $\sigma(\lambda)$ is a simple eigenvalue of $A|_{V}$.

Recall that in a basis adapted to $\ker p$, (possibly after replacing $A'$ with a finite power) the matrix $A'$ has the form 
\[A'=\left(\begin{array}{cc} \Id & * \\ 0 & A \end{array}\right).\]
Since each $\sigma(\lambda)$ is a simple eigenvalue for $A$, it is also a simple eigenvalue for $A'$.

The restriction $A'|_{V'}$ of $A'$ to $V'$ has a similar block upper triangular form, 
\[A'|_{V'}=\left(\begin{array}{cc} \Id & * \\ 0 & A|_V \end{array}\right),\]
where the identity block is the identity on $V'\cap \ker(p)$. 
Since each $\sigma(\lambda)$ is an eigenvalue for $A|_{V}$, it is also an eigenvalue for $A'|_{V'}$. In other words, the eigenvector for the eigenvalue $\sigma(\lambda)$ of $A'$ lies in $V'$. 
\end{proof}

\begin{lem}
Let $A'$ be as in the previous lemma, and $E\subset V'$ denote the span of the eigenvectors of the $\sigma(\lambda)$. Then $E$ is defined over the trace field of $V$. In particular, $V'$ contains vectors with coordinates in $\bk$. 
\end{lem}

\begin{proof}
Let $\{\lambda=\lambda_1, \ldots, \lambda_k\}$ be the set of $\sigma(\lambda)$, where $\sigma$ is a field automorphism of $\bC$ acting trivially on $\bk$. Define the polynomial $g(x)=\prod_{i=1}^k (x-\lambda_i)$, and note that $g(x)\in \bk[x]$ has coefficients in the trace field. 

Let $f(x)$ denote the characteristic polynomial of $A'$. Then $g(x)$ divides $f(x)$, and moreover since the $\lambda_i$ are simple eigenvalues $g(x)$ is coprime to $f(x)/g(x)$.

The Primary Decomposition Theorem gives that there is a projection onto  $\ker(g(A'))=E$ which is a polynomial in $A'$ with coefficients in $\bk$. In particular, since $A'$ is an integer matrix, this projection has matrix coefficients in $\bk$, and hence the image is defined over $\bk$. 
\end{proof}

\begin{proof}[\textbf{\emph{Proof of Proposition \ref{P:galois}}}] 
We can fix a basis for $V\subset \bC^{2g}$ of vectors with coefficients in $\bk(\cM)$. Since the action on $\bC^{2g}$ is via integer matrices, it follows that the representation $V$ can be defined by matrices with coefficients in $\bk(\cM)$. In particular, it is clear that the trace field of $V$ is a subfield of $\bk(\cM)$. 

The previous lemma gives that $V'$ contains a point $v'$ with coordinates in the trace field $\bk$ of $V$. By the Simplicity Theorem, any invariant subspace of $V'$ is contained in $\ker(p)$. The point $v'$ is not contained in $\ker(p)$, so its orbit under the monodromy representation spans $V'$. Since the monodromy of $H^1_{rel}$ is integral, it follows that $V'$ is spanned by points with coordinates in the trace field of $V$. Hence $V'$ is defined over the trace field, that is $\bk=\bk(\cM)$.

This gives that $\bk(\cM)$ is equal to the trace field of $V$. It is easy to see that the field of definition of $p(T(\cM))$ contains the trace field of $V$ and is contained in $\bk(\cM)$. 
\end{proof}

\begin{lem}
For each embedding $\rho$ of $\bk(\cM)$ into $\bC$, there is a simple subbundle $V_\rho$ of $H^1$. These subbundles are non-isomorphic, and we have that 
\[H^1=\left(\bigoplus_\rho V_\rho\right) \oplus W,\]
 where $W$ is semisimple and does not contain any copies of the $V_\rho$. Both $W$ and $\oplus V_\rho$ are defined over $\bQ$.

 In particular, 
\[\dim_\bC p(T(\cM)) \cdot \deg_\bQ \bk(\cM) \leq 2g.\]
\end{lem}

This lemma is exactly Theorem \ref{T:galois} in the language of representations. To translate, set $\bV_\rho$ to be the flat subbundle corresponding to the representation $V_\rho$, etc. 

\begin{proof}
Each embedding $\rho$ of $\bk(\cM)$ into $\bC$ can be written as the identity embedding composed with some field automorphism $\sigma_\rho\in Gal_\bQ(\bC)$. We set $V_\rho=\sigma_\rho(V)$. Since the representation $\pi_1(\cM)$ to $H^1$ can be defined with integer matrices, it follows that all $V_\rho$ appear in this representation. Furthermore, since $\bk(\cM)$ is the trace field of $V$, all these representations $V_\rho$ have different characters and are hence non-isomorphic. They are simple because a simple representation composed with a field automorphism is again a simple representation. 

Again because the representation $\pi_1(\cM)$ to $H^1$ is integral, we have that the multiplicity of each $V_\rho$ in $H^1$ is the same. The multiplicity of $V$ must be one because there are matrices in the representations coming from closed $g_t$--orbits which have a simple eigenvalue whose eigenvector lies in $V$. 

There are $\deg_\bQ \bk(\cM)$ field embeddings of $\bk(\cM)$ into $\bC$, so the final inequality follows from a dimension count: 
\[\sum_\rho \dim_\bC V_\rho \leq \dim_\bC H^1.\]
\end{proof}


\section{Calculation of the field of definition}\label{S:hol}

\begin{proof}[\textbf{\emph{Proof of Theorem \ref{T:main}}}.]
If $\cM$ is an affine invariant submanifold defined over $\bk$, then $\cM$ contains translation surfaces with period coordinates in $\bk[i]$. This is simply linear algebra: linear equations with coefficients in $\bk$ can be solved over $\bk$. Any translation surface with period coordinates in $\bk[i]$ has holonomy field contained in $\bk$. This shows that $\bk(\cM)$ contains the intersection of the holonomy fields; it remains to show that $\bk(\cM)$ is contained in the holonomy field of each translation surface in $\cM$. 

If a translation surface has holonomy field $\bk$, then up to scaling, something in its $SL(2,\bR)$--orbit has absolute period coordinates in $\bk[i]$. So we will show that if $\cM$ contains a translation surface $(X,\omega)$ with absolute period coordinates in $\bk[i]$, then $\bk(\cM)\subset \bk$. 

The point $p(X,\omega)$ represents a point in $V_{\Id}=V$ with coordinates in $\bk[i]$. By the simplicity of $p(T(\cM))$, the orbit of this point under the monodromy representation spans $V_{\Id}$. Since the monodromy is integral, we see that $V_{\Id}$ is spanned by points with coordinates in $\bk[i]$. It follows that $V_{\Id}$ is defined over $\bk[i]$, that is, $p(\cM)$ is defined over $\bk$. By Theorem \ref{T:galois}, it follows that $\cM$ is defined over $\bk$. 

The bound on the degree of $\bk(\cM)$ follows from the inequality in Theorem \ref{T:galois}. (The interested reader may discover an alternate way to conclude that the degree of the field of definition is at most the genus, using the Closing Lemma and the Simplicity Theorem directly.) 
\end{proof}


\section{Limits of algebraically primitive Teichm\"uller curves}\label{S:curves}


\begin{cor}
Suppose $\cM$ is contained in the minimal stratum in prime genus, and that $\cM$ properly contains an algebraically primitive Teichm\"uller curve. Then $\cM$ is equal to a connected component of the stratum. 
\end{cor}

\begin{proof}
By Theorem \ref{T:main}, $\bk(\cM)$ is contained in the trace field of the algebraically primitive Teichm\"uller curve $\cC\subset \cM$. Since this field has prime degree, either $\bk(\cM)$ has degree equal to the genus, or else $\bk(\cM)=\bQ$. Since the dimension of $\cM$ is greater than 2 (because $\cM$ properly contains a Teichm\"uller curve, and the dimension in question is the complex dimension of the tangent space), the inequality in Theorem \ref{T:galois} forces $\bk(\cM)=\bQ$.

In particular, we have that the representation $\pi_1(\cC)$ to the tangent space of $\cM$ must be stable under $Gal_\bQ(\bC)$. However, the representation $\pi_1(\cC)$ to $H^1$ is the sum of $g$ Galois conjugate irreducible representations, so we see that the tangent space of $\cM$ must be all of $H^1$. It follows that $\cM$ is a connected component of a stratum. 
\end{proof}

\begin{proof}[\textbf{\emph{Proof of Theorem \ref{T:primequi}}}]
The results of Eskin-Mirzakhani-Mohammadi \cite[Thm. 2.3]{EMM} give in particular that given any sequence of distinct closed $SL_2(\bR)$--orbits, some subsequence must equidistribute towards some larger affine invariant submanifold which contains the tail of this subsequence.

For the strata in question, the previous corollary shows that the only affine  invariant submanifold which properly contains an algebraically primitive Teichm\"uller curve is the entire stratum. 
\end{proof}


\bibliography{mybib}{}
\bibliographystyle{amsalpha}
\end{document}